\documentclass[12pt]{amsart}
\usepackage[margin=1in]{geometry}
\usepackage{stmaryrd}
\usepackage{amsfonts,amsmath,amssymb,amsthm,amsaddr}
\usepackage{booktabs}
\usepackage{capt-of}
\usepackage{hyperref}
\usepackage{graphicx}
\usepackage{siunitx}
\def\bm{}

\def\iii{\vert\kern-0.25ex\vert\kern-0.25ex\vert}

\def\B{\mathcal{B}}

\def\R{\mathbb{R}}
\def\k{\kappa}

\def\vel{\bm\beta}
\def\dotdot{\cdot,\,\cdot}

\def\llb{\llbracket}
\def\rrb{\rrbracket}

\newtheorem{thm}{Theorem}
\newtheorem{prop}{Proposition}

\begin{document}

\title[Accuracy of DG with reentrant faces]{A short note on the accuracy of the discontinuous Galerkin method with reentrant faces}
\author[Pazner and Haut]{Will Pazner and Terry Haut}
\address{Center for Applied Scientific Computing, Lawrence Livermore National Laboratory}
\begin{abstract}
We study the convergence of the discontinuous Galerkin (DG) method applied to the advection--reaction equation on meshes with reentrant faces.
On such meshes, the upwind numerical flux is not smooth, and so the numerical integration of the resulting face terms can only be expected to be first-order accurate.
Despite this inexact integration, we prove that the DG method converges with order $\mathcal{O}(h^{p+1/2})$, which is the same rate as in the case of exact integration.
Consequently, specialized quadrature rules that accurately integrate the non-smooth numerical fluxes are not required for high-order accuracy.
These results are numerically corroborated on examples of linear advection and discrete ordinates transport equations.
\end{abstract}
\maketitle

\section{Introduction}

The discontinuous Galerkin (DG) method is a high-order finite element method that is well-suited for advection-dominated problems on unstructured meshes.
On general quasi-uniform meshes, the DG discretization of the advection--reaction equation was shown to have $L^2$ error of order $\mathcal{O}(h^{p+1/2})$, where $h$ is the mesh size, and $p$ is the degree of polynomial approximation \cite{Brezzi2004,Johnson1986}.
Although this estimate is suboptimal, it was demonstrated that on so-called \textit{Peterson meshes} that it is sharp \cite{Peterson1991}.
In practice, it is often observed that the DG method converges with optimal $\mathcal{O}(h^{p+1})$ order of accuracy.
Indeed, on several special classes of meshes, the DG method was proven to have optimal $L^2$ order of accuracy for the advection--reaction equation \cite{Richter1988,Cockburn2008}.

The proofs of accuracy of the DG method typically assume exact integration of the volumetric and surface integrals that appear in the formulation.
This is usually not an obstacle, since the integrands are polynomial functions for which well-known quadrature rules may be used.
Even in the case where inexact quadrature rules are used (for example, in the case of the discontinuous Galerkin spectral element method), the quadrature error itself is high-order, and does not pose a barrier to the high-order convergence of the method \cite{Canuto2007}.
The impact of inexact quadrature rules on the accuracy of DG discretizations for hyperbolic conservation laws was also considered in \cite{Huang2016}.
However, if the standard upwind numerical flux is used, and if the mesh contains so-called \textit{reentrant faces} (that is, faces for which the sign of $\vel\cdot\bm n$ changes, where $\vel$ is the velocity field, and $\bm n$ is the normal vector), then the resulting integrand possesses only $C^0$ continuity, and high-order convergence of the integral using quadrature rules can no longer be expected.
For a fixed Gaussian quadrature rule, the surface integral of the numerical flux can only be expected to converge to the true integral with order $\mathcal{O}(h)$, independent of degree of polynomial approximation $p$.
Reentrant faces occur naturally when the velocity field is spatially varying, or when the mesh contains curved elements \cite{Wareing2001}.
Such meshes are important both in order to accurate resolve complex geometries, and in the context of high-order Lagrangian hydrodynamics \cite{Dobrev2012}.
There are numerous important applications of DG discretizations on meshes with reentrant faces, including thermal radiative transfer, describing matter-radiation interaction on high-order meshes obtained from Lagrangian hydrodynamics \cite{Haut2019}, $S_N$ transport on curved meshes \cite{Warsa2002,Wareing2001}, monotonicity-preserving advection-based remap \cite{Anderson2014}, incompressible fluid flow in vorticity--streamfunction formulation \cite{Liu2000}, among others.
Each of these problems involves the DG discretization of advection-type equations on meshes with reentrant faces; for such applications, it is important to maintain high-order accuracy while using the upwind flux that enables efficient sweeping algorithms.
In practice, the predicted order of accuracy of the method is still typically observed, despite the quadrature error.
In this paper, we prove that the standard $\mathcal{O}(h^{p+1/2})$ error estimates still hold in the case of inexact (low-order) integration of the upwind numerical fluxes.
Consequently, we show that specialized quadrature rules that accurately integrate the non-smooth upwind numerical fluxes are not required to attain high-order accuracy for these problems.
Numerical examples of the linear advection--reaction and discrete ordinates radiative transfer equations are used to corroborate the analytical results.

\section{Model problem and discretization}
\label{sec:model-problem}

Let $u \in H^{p+1}(\Omega)$ be the exact solution to the steady linear advection--reaction equation
\begin{equation} \label{eq:adv}
   \nabla\cdot(\vel u) + cu = f,
\end{equation}
in spatial domain $\Omega \subseteq \R^d$, where $\vel : \R^d \to \R^d$ is a prescribed velocity field.
We make the standard assumption (cf.\ \cite{Friedrichs1958,Ern2006}) that there is a positive constant $c_0$ such that
\[
   c(\bm x) + \frac{1}{2} \nabla \cdot \vel(\bm x) \geq c_0 > 0.
\]

Given a mesh $\mathcal{T}_h$ of $\Omega$, we define the DG bilinear form $\B(\dotdot)$ by
\begin{align}
   \label{eq:dg-form}
   \B(u_h, v_h) &= a(u_h, v_h) + b(u_h, v_h) + s(u_h, v_h), \\
   \label{eq:a}
   a(u_h, v_h) &= -\int_\Omega u_h \vel \cdot \nabla_h v_h \, d\bm x + \int_\Omega c u_h v_h \, d\bm x, \\
   \label{eq:b}
   b(u_h, v_h) &= \int_\Gamma \vel \{ u_h \} \cdot \llb v_h \rrb \, ds, \\
   \label{eq:s}
   s(u_h, v_h) &= \int_\Gamma b_0 \llb u_h \rrb \cdot \llb v_h \rrb \, ds.
\end{align}
In the above, $\nabla_h$ denotes the \textit{broken gradient}, evaluated element-by-element, $\Gamma$ denotes the the union of boundaries of elements in $\mathcal{T}_h$, and $b_0$ is a \textit{stabilization function} that must be chosen appropriately.
The form $s(\dotdot)$ is required in order to obtain a stable scheme.
The function $b_0$ is chosen to satisfy
\begin{equation} \label{eq:b0-theta0}
   b_0 \geq \theta_0 | \vel \cdot \bm n | \qquad\text{for some $\theta_0 > 0$.}
\end{equation}
Note that if the stabilization function is chosen as $b_0 = \frac{1}{2} | \vel\cdot\bm n |$, then \eqref{eq:dg-form} is equivalent to the standard \textit{upwind} bilinear form,
\[
   \B_{\textit{upw}}(u_h, v_h)
   = -\int_\Omega u_h \vel \cdot \nabla_h v_h \, d\bm x
   + \int_\Omega c u_h v_h \, d\bm x
   + \int_\Gamma \widehat{u_h} \vel \cdot \llb v_h \rrb \, ds,
\]
where the term $\widehat{u_h}$ denotes the \textit{upwind numerical flux}, defined on the interface between neighboring elements $\k^-$ and $\k^+$ by
\[
   \widehat{u_h} = \begin{cases}
      u_h^- \quad& \text{if $\vel \cdot \bm n^- \geq 0$}, \\
      u_h^+ \quad& \text{otherwise},
   \end{cases}
\]
where $n^-$ points outward from $\k^-$.
Likewise, if $b_0 = 0$ (and so $s(\dotdot) = 0$), then \eqref{eq:dg-form} is equivalent to using the mean value flux, for which the DG scheme is stable only in the $L^2$ norm (and hence can result in highly oscillatory solutions).

In practice, the integrals in the forms \eqref{eq:a}--\eqref{eq:s} are typically computed or approximated using numerical quadrature.
We assume that the quadrature used is sufficiently accurate to integrate the integrals in \eqref{eq:a} and \eqref{eq:b} exactly.
The integrands in \eqref{eq:a} and \eqref{eq:b} are piecewise polynomials whose degree depends on the DG finite element space, and the degree of polynomial approximation used for the mesh geometry, and so Gaussian quadrature rules can efficiently compute these integrals.
If $\vel\cdot\bm n$ does not change sign, then \eqref{eq:s} will also be integrated exactly with the same quadrature.
As a consequence, if the mesh contains no reentrant faces, then the quadrature is exact for all terms in \eqref{eq:dg-form}.
However, on a reentrant face, the sign of $\vel\cdot\bm n$ changes, and so the stabilization function $b_0$ is not smooth.
In this case, generally speaking, even for a Gaussian quadrature rule with high degree of precision, the integrals will not be computed exactly.
Moreover, because $b_0$ is not smooth, the quantities computed with numerical quadrature will not even be \textit{high-order accurate}.
We can only expect the quadrature approximation of the stabilization term to be first-order accurate on a reentrant face, regardless of the degree of polynomial approximation or quadrature.

In what follows, we will prove stability and high-order accuracy of the DG method with inexact integration of the upwind term, \textit{not relying on high-order accuracy of the quadrature approximation}.

The following arguments are closely related to those of Brezzi et al.~\cite{Brezzi2004}, however we work with a discrete DG norm that is defined in terms of a quadrature rule.
On a given edge $e \in \Gamma$, we write the quadrature approximation as
\[
   \int_e \bm\varphi(\bm x) \cdot \bm\psi(\bm x) \, ds \approx I_e [ \bm\varphi, \bm\psi ] := \sum_{i=1}^{n_q} w_{i,e} \bm\varphi(\bm x_{i,e}) \cdot \bm\psi(\bm x_{i,e}).
\]
It is clear that $I_e[\dotdot]$ defines a semi-definite form that induces a seminorm
\[
   | \bm\varphi |_{e}^2 := I_e(\bm \varphi, \bm \varphi).
\]
We also define $I_e(\dotdot)$ and $| \cdot |_e$ similarly for scalar arguments.
Let the form $\hat{s}(\dotdot)$ be defined as the approximation to $\eqref{eq:s}$ with quadrature,
\[
   \hat{s}(u_h, v_h) = \sum_{e\in\Gamma} I_e[b_0 \llb u_h \rrb, \llb v_h \rrb],
\]
and let $\widehat{\B}(\dotdot)$ denote the bilinear form with inexact integration, defined by
\[
   \widehat{\B}(u_h, v_h) = a(u_h, v_h) + b(u_h, v_h) + \hat{s}(u_h, v_h).
\]
We remark again that quadrature rule is taken to be sufficiently accurate that the forms $a(\dotdot)$ and $b(\dotdot)$ can be integrated exactly, and so there is no need to introduce their corresponding approximations.
We define the DG norm with quadrature $\iii v \iii$ by
\[
   \iii v \iii^2 = \| v \|_0^2 + \sum_{e\in\Gamma} | b_0^{1/2} \llb v \rrb |_e^2.
\]
This is the natural norm in which we can show stability of the form $\widehat{\B}(\dotdot)$.

\begin{prop}
   The bilinear form $\widehat{\B}(\dotdot)$ is stable with respect to the norm $\iii\cdot\iii$,
   \[
      \widehat{\B}(v_h, v_h) \gtrsim \iii v_h \iii, \qquad \text{for all $v_h \in V_h$}.
   \]
\end{prop}
\begin{proof}
   From \cite{Brezzi2004}, we have
   \[
      a(v_h, v_h)
      = \int_\Omega \left( \frac{1}{2} \nabla \cdot \vel + c \right) v_h^2 \,d\bm x
      - \frac{1}{2} \int_\Gamma \vel \cdot \llb v_h^2 \rrb \, ds.
   \]
   Similarly, using the identity $\vel \{ v_h \} \cdot \llb v_h \rrb = \frac{1}{2} \vel \cdot \llb v_h^2 \rrb$ we see
   \[
      b(v_h, v_h) = \frac{1}{2} \int_\Gamma \vel \cdot \llb v_h^2 \rrb \, ds.
   \]
   Therefore,
   \[
      a(v_h, v_h) + b(v_h, v_h)
      = \int_\Omega \left( \frac{1}{2} \nabla \cdot \vel + c \right) v_h^2 \,d\bm x
      \geq c_0 \| v_h \|_0^2.
   \]
   It is easy to see that
   \[
      \hat{s}(v_h, v_h)
      = \sum_{e\in\Gamma} I_e[b_0^{1/2} \llb v_h \rrb, b_0^{1/2} \llb v_h \rrb ]
      = \sum_{e\in\Gamma} | b_0^{1/2} \llb v_h \rrb |_e^2.
   \]
   Consequently, $\widehat{B}(v_h, v_h) \gtrsim \iii v_h \iii^2$.
\end{proof}

\begin{thm} \label{thm:conv}
   Let $u$ denote the exact solution to \eqref{eq:adv}, and let $u_h \in V_h$ satisfy $\widehat{B}(u_h, v_h) = \ell(v_h)$ for all $v_h \in V_h$, where the linear form is defined by $\ell(v_h) = \int_\Omega f v_h \, dx$.
   Then,
   \[
      \iii u - u_h \iii \lesssim h^{p+1/2} \| u \|_{p+1}.
   \]
\end{thm}
\begin{proof}
   \def\P{\Pi}
   Let $\P$ denote the $L^2$ projection onto $V_h$, and define
   \[
      \eta = u - \P u, \qquad \delta = u_h - \P u.
   \]
   Let $e = u - u_h$ denote the discretization error.
   Clearly $e = \eta - \delta$, and so we bound
   $
      \iii e \iii \leq \iii \eta \iii + \iii \delta \iii.
   $
   We first claim that
   \begin{equation} \label{eq:eta-ineq}
      \iii \eta \iii \lesssim h^{p+{1/2}} \| u \|_{p+1}.
   \end{equation}
   Note that $b_0 \leq c_e$ on each edge $e \in \partial \k$ for some constant $c_e$, and so
   \begin{align*}
      | b_0^{1/2} \llb u - \P u \rrb |_e^2
      &\leq c_e | \llb u - \P u \rrb |_e^2
      = c_e | \llb \P u \rrb |_e^2 \\
      &= c_e \| \llb \P u \rrb \|_{0,e}^2
      = c_e \| \llb u - \P u \rrb \|_{0,e}^2.
   \end{align*}
   Inequality \eqref{eq:eta-ineq} then follows by combining the above estimate with a standard result on the accuracy of the $L^2$ projection and a trace inequality.

   It remains to bound $\iii \delta \iii$.
   From the stability estimate, we have
   \begin{equation} \label{eq:delta-stability}
      \widehat{\B}(\delta, \delta) \gtrsim \iii \delta \iii^2.
   \end{equation}
   Furthermore,
   \[
      \widehat{\B}(\delta, \delta)
      = \widehat{\B}(u_h - \P u, u_h - \P u) + \widehat{\B}(u - u_h, u_h - \P u)
      = \widehat{\B}(\eta, \delta),
   \]
   since Galerkin orthogonality still holds in the approximate bilinear form $\widehat{\B}(\dotdot)$.
   An argument from \cite{Brezzi2004} shows that
   \begin{equation}
      \label{eq:a-hp}
      a(\eta, \delta)
      \lesssim h^{p+1} \| \delta \|_0 \| u \|_{p+1}.
   \end{equation}
   Additionally,
   \[
      b(\eta, \delta)
      = \int_\Gamma \vel \{ \eta \} \cdot \llb \delta \rrb \, ds
      = \sum_{e\in\Gamma} I_e [ \vel \{ \eta \}, \llb \delta \rrb].
   \]
   Recalling the definition of $\theta_0$ from \eqref{eq:b0-theta0}, we have
   \[
      \vel \{ \eta \} \cdot \llb \delta \rrb
      \leq \left| \vel \cdot \bm n \{ \eta \} (\delta^- - \delta^+) \right|
      \leq \frac{1}{\theta_0} b_0 \left| \{ \eta \} (\delta^- - \delta^+) \right|,
   \]
   and so,
   \begin{align*}
      I_e [ \vel \{ \eta \}, \llb \delta \rrb]
      &= \sum_{i=1}^{n_q} w_i \vel(\bm x_i) \{ \eta(\bm x_i) \} \cdot \llb \delta(\bm x_i) \rrb \\
      &\leq \frac{1}{\theta_0} \sum_{i=1}^{n_q} w_i b_0(\bm x_i) \left| \{ \eta(\bm x_i) \} (\delta^-(\bm x_i) - \delta^+(\bm x_i)) \right| \\
      &= \frac{1}{\theta_0} I_e [b_0^{1/2} |\{\eta \}|, b_0^{1/2} | \delta^- - \delta^+ | ].
   \end{align*}
   By the Cauchy--Schwarz inequality,
   \[
      \frac{1}{\theta_0} I_e [b_0^{1/2} | \eta |, b_0^{1/2} | \delta^- - \delta^+ | ]
      \leq \frac{1}{\theta_0} \left| b_0^{1/2} \{ \eta \} \right|_e \left| b_0^{1/2} \llb \delta \rrb \right|_e.
   \]
   Furthermore,
   \[
      \hat{s}(\eta, \delta)
      = \sum_{e\in\Gamma} I_e [ b_0^{1/2} \llb \eta \rrb, b_0^{1/2} \llb \delta \rrb ]
      \leq \sum_{e\in\Gamma} \left| b_0^{1/2} \llb \eta \rrb \right|_e \left| b_0^{1/2} \llb \delta \rrb \right|_e,
   \]
   and so,
   \begin{equation} \label{eq:b-plus-s}
      b(\eta, \delta) + \hat{s}(\eta, \delta)
      \leq \sum_{e\in\Gamma} \left(
         \frac{1}{\theta_0} \left| b_0^{1/2} \{ \eta \} \right|_e
         + \left| b_0^{1/2} \llb \eta \rrb \right|_e
      \right) \left| b_0^{1/2} \llb \delta \rrb \right|_e.
   \end{equation}
   Applying a trace inequality,
   \[
      | b_0^{1/2} \{ \eta \} |_e
      \leq c_e \| \{ \eta \} \|_{0,e}
      \lesssim h^{p+1/2} \| u \|_{p+1},
   \]
   and
   \[
      | b_0^{1/2} \llb \eta \rrb |_e
      \leq c_e \| \llb \eta \rrb \|_{0,e}
      \lesssim h^{p+1/2} \| u \|_{p+1}.
   \]
   Therefore, another application of the Cauchy--Schwarz inequality to \eqref{eq:b-plus-s} results in
   \begin{equation} \label{eq:b-plus-s-hp}
      b(\eta, \delta) + \hat{s}(\eta, \delta)
      \lesssim h^{p+1/2} \| u \|_{p+1} \left( \sum_e \left| b_0^{1/2} \llb \delta \rrb \right|_e^2 \right)^{1/2}.
   \end{equation}
   Combining \eqref{eq:a-hp} with \eqref{eq:b-plus-s-hp} and \eqref{eq:delta-stability}, we have
   \[
      \iii \delta \iii^2 \lesssim \widehat{\B}(\eta,\delta) \lesssim h^{p+1/2} \| u \|_{p+1} \iii \delta \iii,
   \]
   and so $\iii u - u_h \iii \lesssim h^{p+1/2} \| u \|_{p+1}$.
\end{proof}

\section{Numerical results}

Ample numerical results in the literature have confirmed the high-order accuracy of the DG method applied to the model first-order problem \cite{Cockburn2001,Hesthaven2008}.
In practice optimal rates of $\mathcal{O}(h^{p+1})$ are typically observed for the $L^2$ error, and the rates of $\mathcal{O}(h^{p+1/2})$ predicted by theory occur much more rarely on specially constructed meshes and velocity fields \cite{Johnson1986,Lesaint1974,Peterson1991}.
In the following examples, we confirm high-order convergence specifically on test cases with reentrant faces, and verify that this property still holds even when the \textit{quadrature error} on these faces is only $\mathcal{O}(h)$.

\subsection{Linear advection--reaction} \label{sec:adv-results}
For a first test case, we consider the linear advection--reaction equation \eqref{eq:adv}.
The velocity field is taken to be $\bm\beta = (-y, x)$ and $c$ is given by $c \equiv 0.1$.
The mesh is an unstructured mesh of the domain $[-1,1]\times[-1,1]$, which is then curved by transforming the nodes according to the mapping $(x,y) \mapsto (x\cos\theta - y\sin\theta, y\cos\theta + x\sin\theta)$, where $\theta = 1.5(x^2-1)(y^2-1)$.
The right-hand side is chosen such that the exact solution is given by $u = \exp(0.1\sin(5.1x - 6.2y) + 0.3\cos(4.3x + 3.4y))$.
The solution and the initial (coarse) mesh are shown in Figure \ref{fig:adv}.
The mesh curvature and non-constant velocity field results in 25 reentrant edges on the coarsest mesh.
The polynomial degree is chosen to be $p=3$, and the mesh is refined uniformly 6 times to empirically study the convergence behavior.
On each mesh, we compute the $L^2$ error $\| u - u_h \|_0$ and the error measured in the DG norm $\iii u - u_h \iii$.
Additionally, we measure the \textit{reentrant quadrature error} $Q$ defined by
\[
   Q = \sum_{e\in\Gamma} Q_e := \sum_{e\in\Gamma} \max_{i,j} \left|
      \int_e b_0 \llb \phi_i \rrb \llb \phi_j \rrb \, d\bm x
      - I_e[b_0 \llb \phi_i \rrb, \llb \phi_j \rrb ] \right|,
\]
where $\phi_i$ denotes the DG basis functions.
By accuracy of the quadrature rule, $Q_e = 0$ for any edge $e$ which is not reentrant.
However, on reentrant edges $b_0$ is not smooth, and we expect $Q$ to scale only as $\mathcal{O}(h)$.
Table \ref{tab:conv} displays the convergence results for this problem.
As predicted by the theory, $Q = \mathcal{O}(h)$ and $\iii u - u_h \iii = \mathcal{O}(h^{p+1/2})$.
As is commonly observed in practice, we obtain optimal-order $L^2$ convergence for this problem, $\| u - u_h \|_0 = \mathcal{O}(h^{p+1})$.
This example illustrates the result of Theorem \ref{thm:conv}, guaranteeing high-order convergence of the DG method even when the numerical integration of the upwind term is only first-order accurate.

\begin{table}[b]
   \centering
   \caption{Convergence results for advection--reaction equation on a mesh with reentrant edges.
   $\mathcal{O}(h^{p+1})$ convergence is observed for the $L^2$-norm error, and $\mathcal{O}(h^{p+1/2})$ convergence is observed for the DG-norm error, while $\mathcal{O}(h)$ convergence is observed for the reentrant quadrature error $Q$.}
   \label{tab:conv}
   % \vspace{\abovecaptionskip}
   \begin{tabular}{r|S[table-format=4.0]cc|lc|lc}
      \toprule
      \# DOFs & {\# reentrant} & $Q$ & Rate & $\|u-u_h\|_0$ & Rate & $\iii u - u_h \iii$ & Rate \\
      \midrule
      1{,}904 & 25 & $2.25\times10^{-3}$ & --- & $8.88\times10^{-3}$ & --- & $4.60\times10^{-2}$ & ---\\
      7{,}616 & 52 & $1.33\times10^{-3}$ & 0.76 & $7.47\times10^{-4}$ & 3.57 & $5.63\times10^{-3}$ & 3.03\\
      30{,}464 & 105 & $6.66\times10^{-4}$ & 1.00 & $5.71\times10^{-5}$ & 3.71 & $5.53\times10^{-4}$ & 3.35\\
      121{,}856 & 211 & $3.21\times10^{-4}$ & 1.05 & $3.36\times10^{-6}$ & 4.09 & $5.20\times10^{-5}$ & 3.41\\
      487{,}424 & 421 & $1.45\times10^{-4}$ & 1.15 & $1.91\times10^{-7}$ & 4.14 & $4.73\times10^{-6}$ & 3.46\\
      1{,}949{,}696 & 845 & $7.42\times10^{-5}$ & 0.96 & $1.08\times10^{-8}$ & 4.15 & $4.25\times10^{-7}$ & 3.48\\
      7{,}798{,}784 & 1693 & $3.67\times10^{-5}$ & 1.02 & $6.37\times10^{-10}$ & 4.08 & $3.79\times10^{-8}$ & 3.49\\
      \bottomrule
   \end{tabular}
\end{table}
   % \vspace{\floatsep}

\begin{figure}
   \newdimen\advheight
   \settoheight{\advheight}{%
      \includegraphics[width=1.7in]{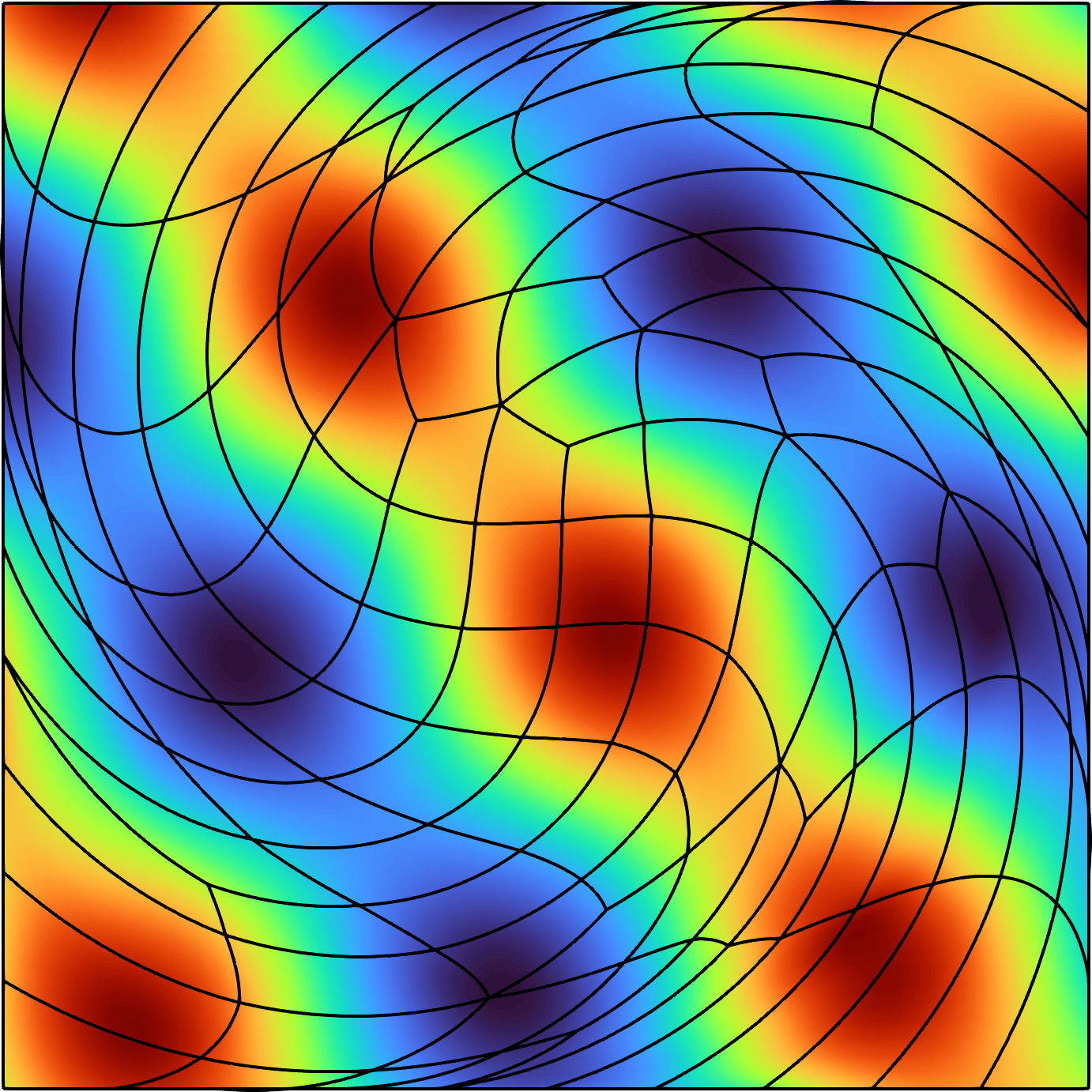}
   }%
   \includegraphics[width=1.7in]{fig/adv0}
   \hspace{2em}
   \includegraphics[width=1.7in]{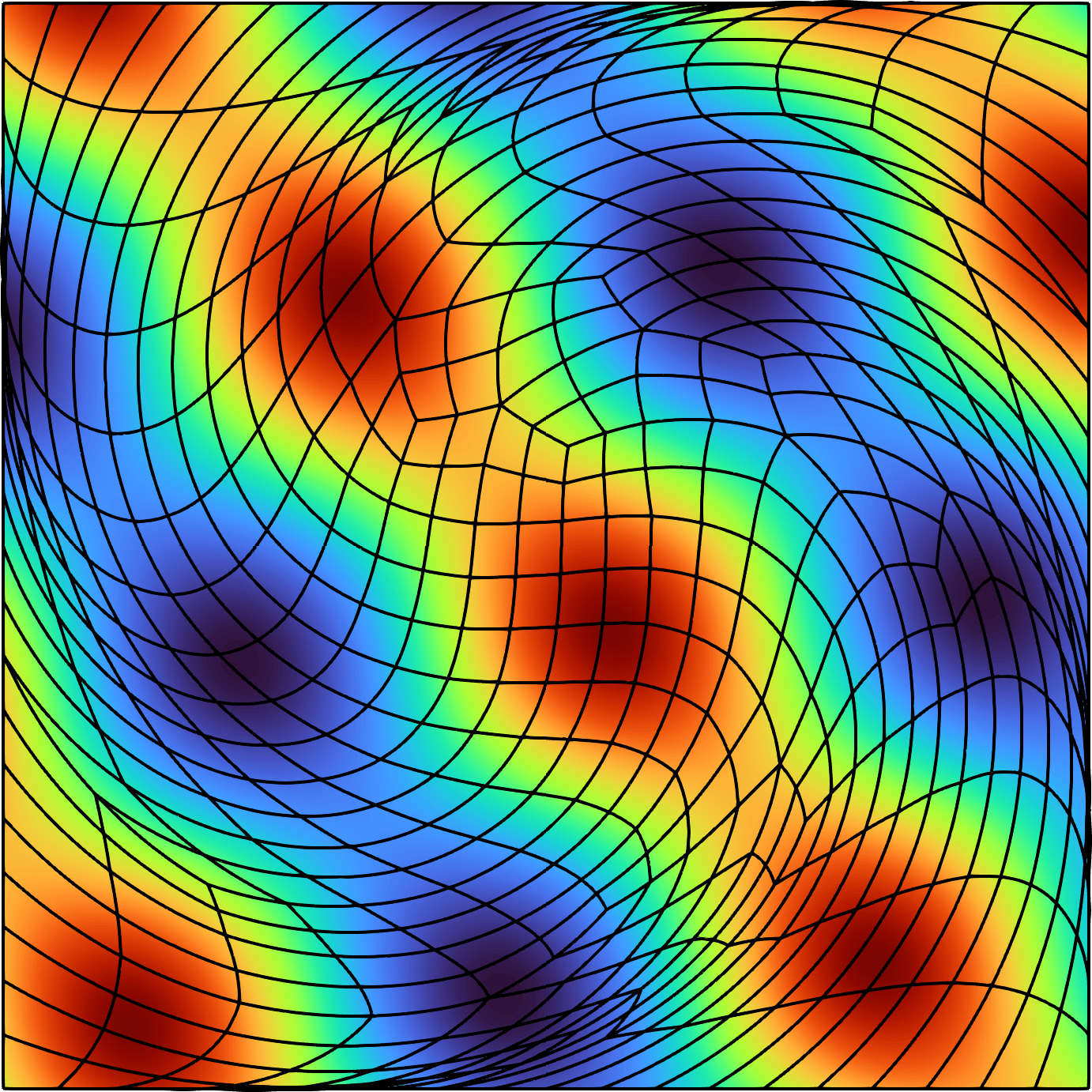}
   \hspace{1em}
   \includegraphics[height=\advheight]{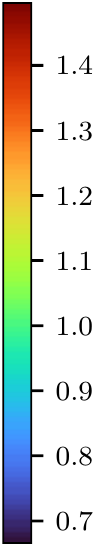}
   % \vspace{\abovecaptionskip}
   \captionof{figure}{Left: initial (coarse) mesh and numerical solution of advection--reaction problem. Right: one uniform refinement.}
   \label{fig:adv}
\end{figure}

\subsection{High-order discrete ordinates transport}
As a more challenging and physically relevant numerical example, we consider the linear, steady-state, monoenergetic Boltzmann equation for the angular flux $\psi_j$ in discrete-ordinate directions $\Omega_j \in \mathbb{S}^2$, $j=1,\ldots,N_\Omega$ (cf.~\cite{Lewis1984}):
\begin{equation}\label{eq:linear-sn}
   \Omega_j \cdot \nabla\psi_j + \sigma_{t}\psi_j
   = \frac{\sigma_s}{4\pi}\varphi + q_j,
\end{equation}
where $\varphi = \sum_{j=1}^{N_\Omega} w_{j}\psi_{j}$ denotes the scalar flux.
The quadrature directions $\Omega_j\in\mathbb{S}^2$ and weights $w_j$ are chosen to exactly integrate all spherical harmonics up to a given degree on $\mathbb{S}^2$.
In the above, $\sigma_t$ is the total opacity, and $\sigma_s$ is the scattering opacity, both of which are non-negative and are determined by the material properties.
Obtaining high-order convergence for this equation is a requisite step for developing high-order accurate numerical methods for the nonlinear thermal radiative transfer equations for the frequency integrated specific intensity.
Equation (\ref{eq:linear-sn}) is discretized using the high-order DG method described in Section \ref{sec:model-problem}; see also \cite{Haut2019} for more details.

We assess the accuracy of the DG method applied to \eqref{eq:linear-sn} using a manufactured solution.
We use polynomial degree $p=3$ on an isoparametric curved mesh obtained from the Lagrangian hydrodynamics simulation of a triple-point shock problem \cite{Dobrev2012}.
The angular discretization uses an $S_{10}$ level-symmetric quadrature set, and results in $N_\Omega = 120$ quadrature directions on the sphere \cite{Lewis1984}.
On curved meshes such as those originating from Lagrangian hydrodynamics, reentrant edges are often impossible to avoid; the coarsest version of the triple-point mesh gives rise to an average of 42 reentrant edges per direction $\Omega_j$.

We define the scattering opacity $\sigma_s(x) = 4/5$ and total opacity $\sigma_t(x) = x^2 + y^2 + 1$.
Then, the source functions $q_j$ and inflow boundary conditions are chosen so that the solution to \eqref{eq:linear-sn} is given by
\begin{equation} \label{eq:psi_j}
   \psi_j(x) =
   \left(\Omega_{j,1}^2 + \Omega_{j,2}\right)
   \left(
      \frac{x^2+y^2+1}{2} + \cos\left( 3 \frac{x+y}{2} \right)
   \right).
\end{equation}
The angular quadrature exactly integrates spherical harmonics $Y_{l,m}$ of degree $l \leq 2$, so the exact scalar flux $\varphi$ is given by
$
   \varphi(x) = \frac{4 \pi}{3}
   \left(
      \frac{x^2+y^2+1}{2} + \cos\left( 3 \frac{x+y}{2} \right)
   \right).
$
The coarsest mesh and scalar flux are shown in Figure \ref{fig:sn}.
In Table \ref{tab:sn} we show the $L^2$ and DG norm error for the scalar flux on a sequence of meshes obtained by uniformly refining the coarsest mesh five times.
Consistent with the results from Theorem \ref{thm:conv} and Section \ref{sec:adv-results}, we observe high-order $\mathcal{O}(h^{p+1/2})$ convergence for this problem in the DG norm.
As in the case of linear advection, we observe $\mathcal{O}(h^{p+1})$ convergence in the $L^2$ norm.

\begin{table}
   \centering
   \begin{minipage}[c]{0.52\linewidth}
      \centering
      \small
      \caption{Convergence of the solution to the neutral particle transport equation on the triple-point mesh.}
      \label{tab:sn}
      \begin{tabular}{r|cc|lc}
         \toprule
         \# DOFs
            & $\| \varphi-\varphi_h \|_{0}$ & Rate
            & \multicolumn{1}{c}{$\iii \varphi-\varphi_h \iii$} & Rate \\
         \midrule
         5{,}376 & $9.40\times10^{-1}$ & --- & $6.07\times10^{0}$ & ---\\
         21{,}504 & $2.19\times10^{-2}$ & 5.42 & $2.03\times10^{-1}$ & 4.90\\
         86{,}016 & $8.62\times10^{-4}$ & 4.67 & $1.04\times10^{-2}$ & 4.29\\
         344{,}064 & $4.29\times10^{-5}$ & 4.33 & $8.41\times10^{-4}$ & 3.63\\
         1{,}376{,}256 & $2.80\times10^{-6}$ & 3.94 & $8.79\times10^{-5}$ & 3.26\\
         5{,}505{,}024 & $1.82\times10^{-7}$ & 3.94 & $8.36\times10^{-6}$ & 3.39\\
         \bottomrule
      \end{tabular}
   \end{minipage}
   \hspace*{\fill}
   \begin{minipage}[c]{0.45\linewidth}
      \centering
      \newdimen\snheight
      \def\snwidth{0.5}
      \settoheight{\snheight}{%
         \includegraphics[width=\snwidth\linewidth]{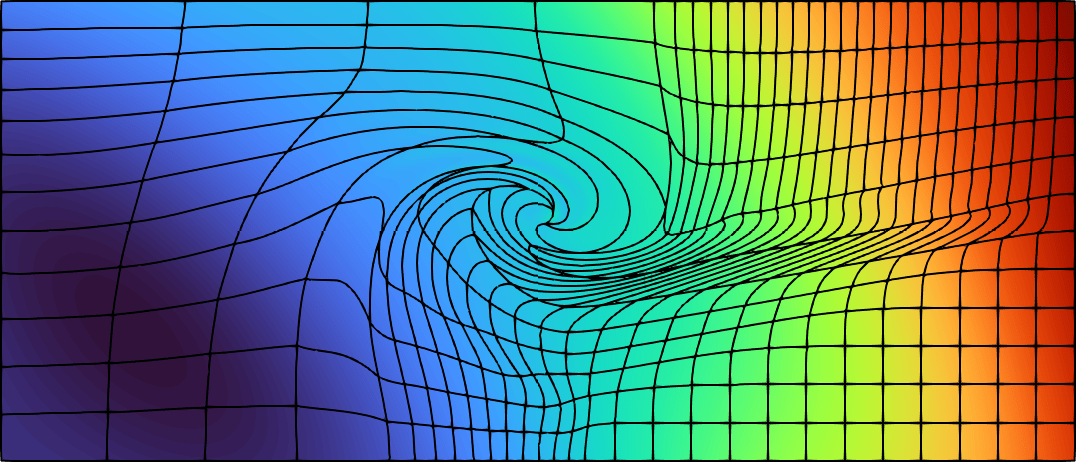}
      }%
      \resizebox{\linewidth}{!}{%
      \includegraphics[width=\snwidth\linewidth]{fig/sn}%
      \hspace{0.1em}
      \includegraphics[height=\snheight]{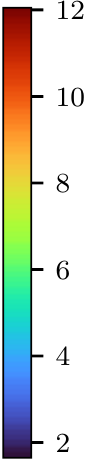}%
      }
      \vspace{0pt}
      \captionof{figure}{Scalar flux $\varphi$ on the triple-point mesh.}
      \label{fig:sn}
   \end{minipage}
\end{table}

\section{Acknowledgements}

{\small This work was performed under the auspices of the U.S.\ Department of Energy by Lawrence Livermore National Laboratory under Contract DE-AC52-07NA27344 (LLNL-JRNL-821642).
This document was prepared as an account of work sponsored by an agency of the United States government.
Neither the United States government nor Lawrence Livermore National Security, LLC, nor any of their employees makes any warranty, expressed or implied, or assumes any legal liability or responsibility for the accuracy, completeness, or usefulness of any information, apparatus, product, or process disclosed, or represents that its use would not infringe privately owned rights.
Reference herein to any specific commercial product, process, or service by trade name, trademark, manufacturer, or otherwise does not necessarily constitute or imply its endorsement, recommendation, or favoring by the United States government or Lawrence Livermore National Security, LLC.
The views and opinions of authors expressed herein do not necessarily state or reflect those of the United States government or Lawrence Livermore National Security, LLC, and shall not be used for advertising or product endorsement purposes.}

\bibliographystyle{siamplain}
\bibliography{refs}

\end{document}